\documentclass[11pt]{article}
\usepackage{cite}
\usepackage{amsmath,amsthm,amssymb,amsfonts}
\usepackage{amsfonts,amssymb,color}
\usepackage{fullpage}

\newtheorem{theorem}{Theorem}[section]

\newtheorem{lemma}[theorem]{Lemma}
\newtheorem{corollary}[theorem]{Corollary}

\numberwithin{equation}{section}

\title{Spectral strengthening of a theorem on transversal critical graphs}

\author{
Muhuo Liu\thanks{
Department of Mathematics, South China Agricultural University, Guangzhou, 510642, China.} 
\thanks{Research Center for Green Development of Agriculture, South China Agricultural University, Guangzhou, 510642, China.
Email: {\tt liumuhuo@163.com}}\ \
and
Xiaofeng Gu\thanks{
Department of Computing and Mathematics, University of West Georgia, Carrollton, GA 30118, USA. 
\newline Email: {\tt xgu@westga.edu}; Corresponding author}
}

\begin{document}
\date{}
\maketitle
\noindent

\begin{abstract}
A transversal set of a graph $G$ is a set of vertices incident to all edges of $G$. The transversal number of $G$, denoted by $\tau(G)$, is the minimum cardinality of a transversal set of $G$. A simple graph $G$ with no isolated vertex is called $\tau$-critical if $\tau(G-e) < \tau(G)$ for every edge $e\in E(G)$. For any $\tau$-critical graph $G$ with $\tau(G)=t$, it has been shown that $|V(G)|\le 2t$ by Erd\H{o}s and Gallai and that $|E(G)|\le {t+1\choose 2}$ by Erd\H{o}s, Hajnal and Moon. Most recently, it was extended by Gy\'arf\'as and Lehel to $|V(G)| + |E(G)|\le {t+2\choose 2}$.
In this paper, we prove stronger results via spectrum. Let $G$ be a $\tau$-critical graph with $\tau(G)=t$ and $|V(G)|=n$, and let $\lambda_1$ denote the largest eigenvalue of the adjacency matrix of $G$. We show that $n + \lambda_1\le 2t+1$ with equality  if and only if $G$ is $tK_2$,   $K_{s+1}\cup (t-s)K_2$,  or $C_{2s-1}\cup (t-s)K_2$, where $2\leq s\leq t$; and in particular, $\lambda_1(G)\le t$ with equality  if and only if $G$ is $K_{t+1}$. We then apply it to show that for any nonnegative  integer $r$, we have $n\left(r+ \frac{\lambda_1}{2}\right) \le {t+r+1\choose 2}$ and characterize all extremal graphs. This implies a pure combinatorial result that $r|V(G)| + |E(G)| \le {t+r+1\choose 2}$, which is stronger than Erd\H{o}s-Hajnal-Moon Theorem and Gy\'arf\'as-Lehel Theorem. We also have some other generalizations.
\end{abstract}



\section{Introduction}
Throughout the paper, let  $G=(V,E)$ be a simple graph. For terminology  not defined here, we refer readers to \cite{BoMu08}. For any $v\in V(G)$, let $d_G(v)$ or simply $d(v)$ denote the degree of $v$. A $d$-regular graph is a graph with the degree of each vertex equals to $d$. A {\it transversal set} of $G$ is a vertex subset $X\subseteq V(G)$ such that every edge of $G$ is incident to some vertex of $X$. The {\it transversal number} of $G$, denoted by $\tau(G)$, is the minimum cardinality of a transversal set of $G$. A transversal set and the transversal number are also called a {\it vertex covering} and the {\it covering number}, respectively (see\cite{BoMu08}). An {\it independent set} is a set of non-adjacent vertices and the {\it independence number} of $G$, denoted by $\alpha(G)$, is the maximum cardinality of an independent set in $G$. Clearly, $T$ is a transversal set of $G$ if and only if $V(G)-T$ is an independent set of $G$, and therefore $\tau(G) + \alpha(G) = |V(G)|$. A simple graph $G$ with no isolated vertex is called {\it $\tau$-critical} if $\tau(G-e) < \tau(G)$ for every edge $e\in E(G)$. Since $\tau(G-e)\geq \tau(G)-1$, it follows that if $G$ is $\tau$-critical, then  $\tau(G-e)=\tau(G)-1$ for  every edge $e\in E(G)$. Furthermore, $\tau(G)\leq |V(G)|-1$ with equality if and only if $G$ is a complete graph. More properties of $\tau$-critical graphs can be found in Chapter 12 of \cite{LoPl86} by Lov\'asz and Plummer.

In the following, we always assume that $t$ is a positive integer. Let $K_n, C_n, tK_2$ denote the complete graph on $n$ vertices, the cycle on $n$ vertices and $t$ disjoint copies of $K_2$, respectively. For any $\tau$-critical graph $G$ with $\tau(G)=t$, it has been shown that $|V(G)|\le 2t$ and conjectured that $|E(G)|\le {t+1\choose 2}$ by Erd\H{o}s and Gallai~\cite{ErGa61}. The conjecture was settled by Erd\H{o}s, Hajnal and Moon~\cite{EHM64}.

\begin{theorem}[Erd\H{o}s, Hajnal and Moon~\cite{EHM64}]
\label{ehmthm}
For any $\tau$-critical graph $G$ with $\tau(G)=t$, $$|E(G)|\le {t+1\choose 2},$$ where the  equality holds if and only if $G$ is $K_{t+1}$.
\end{theorem}

The theorem was recently extended  by Gy\'arf\'as and Lehel~\cite{GyLe20} with a very short proof using  the sum of order and size.
This combined bound immediately implies Theorem~\ref{ehmthm} since $|V(G)|\ge t+1$.

\begin{theorem}[Gy\'arf\'as and Lehel~\cite{GyLe20}]
\label{gyth}
For any $\tau$-critical graph $G$ with $\tau(G)=t$, $$|V(G)| + |E(G)|\le {t+2\choose 2},$$
with equality if and only if $G$ is $K_{t+1}$, $2K_2$, or $C_5$.
\end{theorem}

We will prove an extension of Theorems~\ref{gyth}, as shown below. Clearly, Theorem~\ref{rvpe} implies Theorems~\ref{ehmthm} and \ref{gyth} when $r=0, 1$, respectively.

\begin{theorem}
\label{rvpe}
Let $r$ be a nonnegative integer. Suppose that $G$ is a $\tau$-critical graph with $\tau(G)=t \ge r$ and $|V(G)|=n$. Then
\begin{equation*}
r|V(G)| + |E(G)|\le {t+r+1\choose 2},
\end{equation*}
with equality if and only if $G$ is $K_{t+1}$ when $r=0$; $K_{t+1}$, $2K_2$, or $C_5$ when $r=1$;
and $rK_2$, $(r+1)K_2$, $C_{2r+1}$, or $C_{2r+3}$ when $r\ge 2$.
\end{theorem}

In fact, we discover stronger results involving the largest eigenvalue of adjacency matrix that strengthen Theorems~\ref{ehmthm}, \ref{gyth} and \ref{rvpe}, with all extremal graphs characterized. This will be done in Section~\ref{sec:spect}, and Theorem~\ref{rvpe} will be proved there. 

In the last section, we make some remarks that our results can be further extended to more general forms and to eigenvalues of signless Laplacian matrix. We start with some combinatorial tools on $\tau$-critical graphs in the next section.

\section{Combinatorial preliminaries}
In this section, we present two useful theorems on $\tau$-critical graphs, and prove a combinatorial lemma, which will play an important role in the proofs of our main results. Theorem~\ref{haj} is the complement version of a theorem of Hajnal~\cite{Hajnal65}. Theorem~\ref{lplem} was discovered by Sur\'anyi~\cite{Sura73}, but can also be found in \cite{LoPl86}.
\begin{theorem}[Hajnal~\cite{Hajnal65}]
\label{haj}
Let $G$ be a $\tau$-critical graph with $\tau(G)=t$ and $|V(G)|=n$. Then the degree $d(v)\le 2t+1-n$ for every $v\in V(G)$.
\end{theorem}
\begin{theorem}[Sur\'anyi~\cite{Sura73}; see also Theorem 12.1.13 of \cite{LoPl86}]
\label{lplem}
Let $G$ be a $\tau$-critical graph and $S$ be an independent set of $G$. Then for any vertex $v\in S$,
$d_G (v)\le |N(S)| -|S| +1$.
\end{theorem}
Theorem~\ref{lplem} actually implies Theorem~\ref{haj}. To see this, notice that for any vertex $v$, we can find a maximum independent set $S$ containing $v$ (by Lemma 12.1.2 of \cite{LoPl86}). Then $N(S)$ is a minimum transversal set of $G$, where $N(S)$ denotes the set of vertices that are adjacent to some vertex of $S$.

\begin{lemma}
\label{extg}
Let $G$ be a $\tau$-critical graph with $\tau(G)=t$ and  $|V(G)|=n$. If $G$ contains  a  $(2t+1-n)$-regular component,
then $G$ is $t K_2$, $K_{s+1}\cup (t-s)K_2$   or $C_{2s-1}\cup (t-s)K_2$, where $2\leq s\leq t$.
\end{lemma}
\begin{proof} Let $d=2t+1-n$ and let   $G'$ be a $d$-regular component of $G$. Let $u\in V(G')$.
Since $G$ is $\tau$-critical, $G'$ is also $\tau$-critical and we suppose that $\tau(G')=s$. Notice that if $X$ is a transversal set of $G'$,
then $V(G')-X$ is an independent set, and thus the independence number $\alpha(G')=p-s$, where $|V(G')|=p$.

Notice that $G-G'$ is $\tau$-critical and so $|V(G-G')|\leq 2\tau(G-G')$, that is, $n-p\le 2(t-s)$. It follows that $2s-p \le 2t-n$.
By Theorem~\ref{haj}, we have $$2s+1-p\leq 2t+1-n =d= d(u)\leq 2s+1-p$$ and thus $$2(t-s)=n-p\,\,\text{and}\,\,d=2s+1-p.$$
It follows that $|V(G-G')| = 2\tau(G-G') = 2(t-s)$.
By Theorem~\ref{haj}, the degree of  each vertex of $G-G'$ is at most $2\tau(G-G')-|V(G-G')|+1=1$, and thus $G-G'=(t-s)K_2$.

To complete the proof, it suffices to show that $G'$ is $K_2$, $K_{s+1}$   or $C_{2s-1}$ for $s\geq 2$.
If $d=1$, then $G'$ is $K_2$ and so the result already holds. Thus, we may suppose that $d\geq 2$ in the following.

 Let  $v_1, v_2,\ldots, v_{d}$  be all  neighbors of $u$, and let  $e_i = uv_i$ for $i=1,2,\ldots, d$.
Since $G'$ is $\tau$-critical, we have $\tau(G'-e_i) = s-1$. In other words, we can find an independent set $S_i$ of $G'-e_i$
with $|S_i| = p-s+1$. By the $\tau$-criticality of $G'$,  we have $u, v_i\in S_i$.
(Otherwise if one of $\{u, v_i\}$ is not in $S_i$, then adding $e_i$ back does not affect the independence
of $S_i$ and thus $S_i$ is also an independent set of $G'$, violating the fact that $\alpha(G')=p-s$.)

Now for all $j\neq i$, $v_j$ is adjacent to $u$ in $G'-e_i$, and thus $v_j\not\in S_i$ since $u\in S_i$.
Let $T_i = S_i - u$. 
Since $T_i$ is an independent set of $G'$ with $|T_i| = p -s$, it turns out that $T_i$ is a maximum independent set of $G'$.

We claim that $T_i\cap T_j =\emptyset$ for all $i\neq j$. If not, then let $w\in T_i\cap T_j$. Since $T_i\cap T_j$ is an independent set,
by Theorem~\ref{lplem}, we have $$2s+1-p=d= d(w) \le |N(T_i\cap T_j)| -|T_i\cap T_j| +1,$$ whence
$$|N(T_i\cap T_j)| \ge |T_i\cap T_j|  + 2s-p=|T_i\cap T_j| +p -|T_i| - |T_j| = p - |T_i\cup T_j|.$$

Observing that $N(T_i \cap T_j)$ is a subset of $V(G')\setminus (T_i \cup T_j)$, we can conclude that every vertex of $G'$ outside of $T_i\cup T_j$ is adjacent to some vertex in $T_i\cap T_j$. This is impossible, since $u$ is not adjacent to any vertex in $T_i\cap T_j$. Thus $T_i\cap T_j =\emptyset$ for all $i\neq j$.

Hence $$p = |V(G')| \ge 1 + \sum_{i=1}^{d} |T_i| = 1 +d(p-s) = 1 + (2s+1-p)(p-s),$$ yielding $$(p-s-1)(2s-1-p)\le 0.$$
Notice that we always have $p\ge s+1$. Hence $p=s+1$ or $2s-1-p\le 0$.

If $p = s+1$, then $$d=2s+1-p=s=p-1$$ and so   $G'$ is $K_{p} = K_{s+1}$.

If $2s-1-p\le 0$, then $$2\leq d= 2s+1 -p\le 2,$$ yielding that   $p=2s-1$, and so  $G'= C_p = C_{2s-1}$.
This completes the proof of the lemma.
\end{proof}

\section{Spectral strengthening}
\label{sec:spect}
In this section, we present spectral analogues that strengthen Theorems~\ref{ehmthm} and \ref{gyth}, and determine all extremal graphs. Theorem~\ref{nlam}, Corollaries~\ref{lam1} and \ref{spect} are the main results, and Theorem~\ref{rvpe} will be proved at the end of the section.

Let $\lambda_1 := \lambda_1(G)$ denote the largest eigenvalue of the adjacency matrix of $G$, which is called the {\it spectral radius} of $G$. The following lemma is well-known (see also Theorem 3.2.1 of \cite{Cd2}).
\begin{lemma}\label{21e}
For any graph $G$ with maximum degree $\Delta$, $$\frac{2|E(G)|}{|V(G)|}\leq \lambda_1\leq \Delta,$$
with the first equality if and only if $G$ is regular, and the second equality if and only if $G$ has a $\Delta$-regular component.
\end{lemma}

Now we are ready to prove a spectral extremal theorem on $\tau$-critical graphs with respect to the sum of the order and spectral radius, which strengthens Theorem~\ref{gyth}.
\begin{theorem}\label{nlam}
For any $\tau$-critical graph $G$ with $\tau(G)=t$ and $|V(G)|=n$, $$n + \lambda_1 \le 2t+1,$$
with equality   if and only if $G$ is  $tK_2$, $K_{s+1}\cup (t-s)K_2$,  or $C_{2s-1}\cup (t-s)K_2$ for $2\leq s\leq t$.
\end{theorem}
\begin{proof}
By Theorem~\ref{haj} and Lemma~\ref{21e}, we have $\lambda_1\leq \Delta \leq 2t +1-n$, and thus $n + \lambda_1 \le   2t+1$. It is easily checked that the equality holds when $G$ is one of these extremal graphs as described. Conversely, if the equality holds, then $\lambda_1 =\Delta =2t +1-n$, and thus Lemma~\ref{21e} implies that $G$ contains a $(2t+1-n)$-regular component.  The result then follows from Lemma~\ref{extg}.
\end{proof}

The following corollary is a simple extremal result with respect to solely spectral radius, and it can be considered as a spectral analogue of Theorem~\ref{ehmthm}.
\begin{corollary}
\label{lam1}
If $G$ is a $\tau$-critical graph with $\tau(G)=t$, then $\lambda_1(G)\le t$ with equality  if and only if $G$ is $K_{t+1}$.
\end{corollary}
\begin{proof}
Since $|V(G)|\ge t+1$, we have $\lambda_1(G)\le t$  by Theorem~\ref{nlam}. The equality holds if and only if $|V(G)| = t+1$,
that is,  if and only if $G$ is $K_{t+1}$.
\end{proof}

The following result looks more general, however, it can be derived from Theorem~\ref{nlam}.
\begin{corollary}
\label{spect}
Let $r$ be a nonnegative integer. Suppose that $G$ is a $\tau$-critical graph with $\tau(G)=t \ge r$ and $|V(G)|=n$. Then
$$n\left(r + \frac{\lambda_1}{2}\right) \le {t+r+1\choose 2},$$  with equality if and only if one of the following holds:
\\$(1)$ $K_{t+1}$  when $r=0$;
\\$(2)$ $K_{t+1}$, $K_{t}\cup K_2$ for $t\ge 2$, or $C_5$, when $r=1$;
\\$(3)$ $rK_2$, $(r+1)K_2$, $K_{t-r+2}\cup (r-1)K_2$, $K_{t-r+1}\cup rK_2$ with $r\leq t-2$,  $C_{2s-1}\cup (r+1-s)K_2$ for $2\leq s\leq r+1$,
or $C_{2s-1}\cup (r+2-s)K_2$ for $2\leq s\leq r+2$,  when $r\ge 2$.
\end{corollary}
\begin{proof}
Since $\lambda_1\le 2t -n+1$  by Theorem \ref{nlam}, we have $n\left(r+ \frac{\lambda_1}{2}\right) \le \frac{n(2t+2r+1 -n)}{2}$.
Since $\frac{n(2t+2r+1 -n)}{2}$ is maximized when $n=t+r$  or $n=t+r+1$, we have
$$n\left(r+ \frac{\lambda_1}{2}\right)\le \frac{(t+r+1)(t+r)}{2} = {t+r+1\choose 2}.$$
If $G$ is one of the graphs listed in the statement, then it is easily checked that the equality holds. Next we suppose that the equality holds.
Then $\lambda_1 = 2t-n+1$ and either $n=t+r$ or $n=t+r+1$. By Theorem~\ref{nlam}, we can deduce the following:

(1) $r=0$. Notice that $n\geq t+1$ by the definition of transversal number, it follows that $n=t+1$. Thus $G$ is $K_{t+1}$.

(2) $r=1$. Then $n=t+1$ or $n=t+2$. By verifying all extremal graphs in Theorem~\ref{nlam}, it follows that $G$ is $K_{t+1}$ if $n=t+1$; and $G$ is $2K_2$ (i.e. $K_2\cup K_2$), $K_t\cup K_2$ for $t\ge 3$, or $C_5$ if $n=t+2$.

(3) $r\ge 2$. Then $n=t+r$ or $n=t+r+1$. Again, it suffices to verify all extremal graphs in Theorem~\ref{nlam}. If $n=t+r$, then $G$ is $tK_2 = rK_2$, $K_{t-r+2}\cup (r-1)K_2$ or $C_{2s-1}\cup (t-s)K_2 =C_{2s-1}\cup (r+1-s)K_2$ for $2\leq s\leq r+1$; and if $n=t+r+1$, then $G$ is $tK_2 = (r+1)K_2$,   $K_{t-r+1}\cup rK_2$ provided  that $r\leq t-2$, or $C_{2s-1}\cup (t-s)K_2 =C_{2s-1}\cup (r+2-s)K_2$ for $2\leq s\leq r+2$.
\end{proof}

In the rest of this section, we present a proof of Theorem~\ref{rvpe}.
\begin{proof}[Proof of Theorem~\normalfont\ref{rvpe}]
By Lemma \ref{21e}, we have $|E(G)|\le n\lambda_1 /2$ and so $r|V(G)| + |E(G)|\le n\left(r + \frac{\lambda_1}{2}\right)$. By Corollary~\ref{spect}, it follows that $r|V(G)| + |E(G)|\le {t+r+1\choose 2}$. The equality holds if and only if the equalities hold in both Lemma \ref{21e} and Corollary~\ref{spect}. In other words, the extremal graphs are those graphs in Corollary~\ref{spect} that are regular. By verifying the graphs listed in Corollary~\ref{spect}, we conclude that the equality holds if and only if $G$ is $K_{t+1}$ when $r=0$; $K_{t+1}$, $2K_2$, or $C_5$ when $r=1$; and $rK_2$, $(r+1)K_2$, $C_{2r+1}$, or $C_{2r+3}$ when $r\ge 2$.
\end{proof}

\section{Remarks}
In this paper, we proved spectral extremal results for $\tau$-critical graphs. The results also imply a combinatorial theorem which is stronger than former results by Erd\H{o}s, Hajnal and Moon~\cite{EHM64} and by Gy\'arf\'as and Lehel~\cite{GyLe20}. We make some remarks to show further generalizations.

\medskip\noindent
{\bf Remark 1.}
In the proof of Corollary~\ref{spect}, notice that $r$ is nonnegative but not necessary to be integral. In fact, $\frac{n(2t+2r+1 -n)}{2}$ is maximized
when $n=(2t+2r+1)/2$, and thus we have $$n\left(r+ \frac{\lambda_1}{2}\right)\le \frac{(2t+2r+1)^2}{8}.$$
The equality is attained only if $(2t+2r+1)/2$ is an integer, i.e., $2r$ is odd. By Theorem \ref{nlam}, it is not hard to show that the extremal graphs
are $K_{t+1}$ for $r=\frac{1}{2}$ as well as $\frac{2r+1}{2}K_2$, $K_{t-r+3/2}\cup (r-\frac{1}{2})K_2$ or $C_{2s-1}\cup \frac{3+2r-2s}{2}K_2$
where $2\leq s\leq r+\frac{3}{2}$, for other values $r$ such that $2r$ is odd.

More general forms like $n\left(r_1 + r_2\frac{\lambda_1}{2}\right)$ and $r_1 |V(G)| + r_2 |E(G)|$ may be done similarly
with more complicated arguments, but we did not put effort on it.

\medskip\noindent
{\bf Remark 2.} The matrix $Q(G)=D(G)+A(G)$ is the {\it signless Laplacian matrix} of $G$, where $D(G)$ is the diagonal degree matrix
and $A(G)$ is the adjacency matrix of $G$. Let $q_1 := q_1(G)$ be the largest eigenvalue of $Q(G)$, which is called the {\it signless Laplacian spectral radius} of $G$. It is known that (see \cite{Cd1}) $2\lambda_1\le q_1\leq 2\Delta$,  where $\Delta$ is the maximum degree of $G$. Furthermore, the second equality holds if and only if $G$ has a $\Delta$-regular component. Thus, with the same method, we can extend Theorem~\ref{nlam}, Corollaries~\ref{lam1} and \ref{spect} by showing that
$$n + \frac{q_1}{2}\le 2t+1,\ \  q_1\le 2t \ \ \text{and}\ \ n\left(r+ \frac{q_1}{4}\right)\le {t+r+1\choose 2},$$
with the same extremal graphs as those in Theorem~\ref{nlam}, Corollaries~\ref{lam1} and \ref{spect}.

\par\medskip
\noindent {\bf Acknowledgements.}
The authors would like to thank two anonymous referees for their valuable comments which lead to an improvement of the original manuscript.
Xiaofeng Gu is partially supported by a grant from the Simons Foundation (522728). 


\end{document}